\documentclass[dvipdfmx,12pt]{amsart}
\usepackage{amsmath,amssymb,amsthm,mathabx,cases,amscd,array,arydshln}
\usepackage{tikz-cd,graphicx}
\usepackage{ytableau}
\makeindex
\numberwithin{equation}{section}

\theoremstyle{plain}
\newtheorem{thm}{Theorem}
\newtheorem{pro}[thm]{Proposition}        
\newtheorem{cor}[thm]{Corollary}
       
\newtheorem{defi}[thm]{Definition}
\newtheorem{rem}[thm]{Remark}
\newtheorem{lem}[thm]{Lemma}


\newcommand{\g}{\mathfrak{g}}
\newcommand{\R}{\mathbb R}
\newcommand{\C}{\mathbb C}
\newcommand{\Z}{\mathbb Z}

\begin{document}
\title[tt*-Toda equations and Quantum cohomology]{Solutions of the tt*-Toda Equations and Quantum Cohomology of Flag Manifolds}
\author{Yoshiki Kaneko}
\date{}
\maketitle
\begin{abstract}
We relate the quantum cohomology of minuscule flag manifolds to the tt*-Toda equations, a special case of the topological-antitopological fusion equations which were introduced by Cecotti and Vafa in their study of supersymmetric quantum field theories.  To do this, we combine the Lie-theoretic treatment of the tt*-Toda equations of Guest-Ho with the  Lie-theoretic description of the quantum cohomology of minuscule flag manifolds from Chaput-Manivel-Perrin and Golyshev-Manivel.
\end{abstract}
\vspace{10pt}
\section{Introduction}
\hspace{1pt}\\
\indent It is well known that solutions of the 2-dimensional Toda equations correspond to primitive harmonic maps into flag manifolds. The tt*-Toda equations provide a special case of the Toda equations; here the harmonic maps can be regarded as generalizations of variations of Hodge structure, or VHS. Certain special solutions illustrate the mirror symmetry phenomenon: For example, according to Cecotti and Vafa \cite{CV}, the generalized VHS for a solution may correspond to the quantum (orbifold) cohomology of a certain K\"{a}hler manifold.\\
\indent To be more precise, there are three aspects of this result. First it is necessary to establish a bijective correspondence between global solutions on \(\C^{*}=\C-\{0\}\) and their ``holomorphic data". Second, this holomorphic data has to be identified with a flat connection of the type used by Dubrovin in the theory of Frobenius manifolds - we call it the Dubrovin connection. Finally, for certain specific solutions, this has to be identified with the Dubrovin connection associated to the (small) quantum cohomology of a specific K\"{a}hler manifold. Guest, Its and Lin have investigated all three aspects in the case of the Lie group type \(A_{n}\) \cite{GIL}.\\
\indent In \cite{GH} the tt*-Toda equations are described for general complex simple Lie algebras. Guest and Ho obtained a correspondence between solutions and the fundamental Weyl alcove. It is expected (but not yet proved beyond the \(A_{n}\) case) that this gives a bijective correspondence between global solutions and points of (a subset of) the fundamental Weyl alcove.\\
\indent This paper is a contribution to the second and third aspects of the generalization of \cite{GIL} to the case of general complex simple Lie algebras. That is, we shall establish a correspondence between the holomorphic data of certain specific solutions of the tt*-Toda equations and the Dubrovin connections of minuscule flag manifolds, based on the Lie-theoretic approach of \cite{GH}. Minuscule flag manifolds are the projectivized weight orbits of minuscule weights (see \cite{CMP}).\\
\indent The quantum cohomology of flag manifolds has been the subject of many articles, especially from the point of view of quantum Schubert calculus.  For Lie-theoretic treatments we mention in particular \cite{FW}. 
The minuscule case has been studied in detail in \cite{CMP}.\\
\indent Golyshev and Manivel \cite{GM} described the quantum cohomology of minuscule flag manifolds in the context of the Satake isomorphism. For geometers, the most familiar example of this is the relation between the cohomology of the Grassmannian and the  exterior powers of the cohomology of projective space. A quantum version of this was established in \cite{GM}.  It depends on a description of the quantum cohomology of a minuscule flag manifold \(G/P_{\lambda_{i}}\) in terms of a family of  Lie algebra elements denoted by \(\sum_{j=1}^{l}e_{-\alpha_{j}}+qe_{\psi}\) of the Lie algebra \(\g\) (see section 2). Namely, quantum multiplication by the generator of the second cohomology of \(G/P_{\lambda_{i}}\) coincides with the action of \(\sum_{j=1}^{l}e_{-\alpha_{j}}+qe_{\psi}\) under the representation whose highest weight is \(\lambda_{i}\).\\
\indent Our main observation is that this element arises from a certain solution of the tt*-Toda equations. In the theory of \cite{GH}, this solution corresponds to the origin of the fundamental Weyl alcove. The Dubrovin connection is then \(d + (1/\lambda)(\sum_{j=1}^{l}e_{-\alpha_{j}}+qe_{\psi})dq/q\).\\
\indent As this solution depends only on \(G\), i.e. it is independent of the choice of minuscule representation of \(G\), we obtain a relation between the quantum cohomology rings of all minuscule flag manifold \(G/P_{\lambda_{i}}\) (for fixed \(G\)).  For Lie groups of type \(A_n\), \(D_n\), \(E_6\), there are several minuscule weights; thus in these cases the same solution of the tt*-Toda equations corresponds to the quantum cohomology of several minuscule flag manifolds. In particular, this means that the tt*-Toda equation gives an explanation for the quantum Satake isomorphism of \cite{GM}.\\
\indent In addition to these tt* aspects, we shall give more concrete statements and more details of the quantum cohomology results, based on the existing literature. We shall show directly how the above statement concerning the action of \(\sum_{j=1}^{l}e_{-\alpha_{j}}+qe_{\psi}\) follows from the  quantum Chevalley formula. Unlike the original proof in \cite{GM}, a case by case argument is not needed for this.

\indent The following are the contents of this paper. First we  review some aspects of the tt*-Toda equations, quantum cohomology and representation theory. In Section 2.1, we prepare  notation and we recall the tt*-Toda equations for general complex simple Lie groups. Then we describe the relationship between a solution and an element of the fundamental Weyl alcove. After that we give the definition of the Dubrovin connection in Section 2.2. In Section 2.3, we review the relations between representations, homogeneous spaces, and cohomology, in particular in the minuscule case. In Section 2.4, we make some observations on minuscule weight orbits. In section 2.5, we state the main theorem of this paper, which gives an explicit relation between the quantum cohomology of a minuscule flag manifold \(G/P_{\lambda_{i}}\) and a particular solution of the tt*-Toda equations for \(G\). We give part of the proof there, and make some comments on the quantum Satake isomorphism. The proof is completed in section 3.\\
\\
\textbf{\it Acknowledgement}: The author would like to thank Prof. Martin Guest for his considerable support. The author would also like to thank Prof. Takeshi Ikeda and Prof. Takashi Otofuji for their useful discussions and comments. The author would also like to thank the members of geometry group at Waseda for their helpful comments and discussions.
\\
\section{Preliminaries}
\hspace{1pt}\\
\indent First of all, we prepare some aspects of the tt*-Toda equations. Then we review some representation theory. We discuss  minuscule weights and irreducible representations. From the Bruhat decomposition, we can obtain a cell decomposition of the projectivized maximal weight orbit, its cohomology and its quantum cohomology \cite{FW}.\\


\subsection{The tt*-Toda Equations}
\hspace{1pt}\\
\indent We explain some theory of the tt*-Toda equations. It is possible to obtain  local solutions through the DPW construction, and a relationship between the space of local solutions and the fundamental Weyl alcove. For more details we refer to the article by Guest and Ho [GH].\\
\indent Let \(G\) be a complex simple simply-connected Lie group of rank \(l\) and \(\g\) be its Lie algebra. We take a Cartan subalgebra \(\mathfrak{h}\) and let \(\g=\mathfrak{h}\oplus\bigoplus_{\alpha\in\triangle}\g_{\alpha}\) be the root decomposition where \(\triangle\) is the set of roots. We choose positive roots \(\triangle^{+}\) and we obtain simple roots \(\Pi=\{\alpha_{1},\cdots,\alpha_{l}\}\). Let \((\;,\;)\) be any positive scalar multiple of the Killing form. This Killing form induces an inner product on \(\mathfrak{h}^{*}\). We also denote this inner product on \(\mathfrak{h}^{*}\) by the same notation \((\;,\;)\). We denote the coroot \(\frac{2\alpha}{(\alpha,\alpha)}\) of \(\alpha\) by \(\alpha^{\vee}\). We define an ordering of the roots by \(\alpha<\beta\) if \(\beta-\alpha\) is positive. \\
\indent We define \(H_{\alpha}\) by \((H_{\alpha},h)=\alpha^{\vee}(h)\) for all \(h\) in \(\mathfrak{h}\). Then we obtain a basis \(H_{\alpha_{1}},\cdots,H_{\alpha_{l}}\) of \(\mathfrak{h}\). We may choose basis vectors \(e_{\alpha}\in{\g_{\alpha}}\) such that  \((e_{\alpha},e_{-\alpha})=\frac{2}{(\alpha,\alpha)}\) for all \(\alpha\in{\triangle}\). Then we have 
\[[e_{\alpha},e_{\beta}]=
\begin{cases}
	0 & \text{if}\;\alpha+\beta\notin{\triangle}\\
	H_{\alpha} & \text{if}\;\alpha+\beta=0 \\
	N_{\alpha+\beta}e_{\alpha+\beta} & \text{if}\;\alpha+\beta\in{\triangle}-\{0\}
\end{cases}
\]
where \(N_{\alpha+\beta}\) is a nonzero complex number. We define \(\epsilon_{i}\) as the basis of \(\mathfrak{h}\) which is dual to \(\alpha_{i}\), that is \(\alpha_{i}(\epsilon_{j})=\delta_{i,j}\). We denote  the highest root by \(\psi:=\sum_{j=1}^{l}q_{j}\alpha_{j}\) and the Coxeter number by \(s:=1+\sum_{j=1}^{l}q_{j}\).\\
\indent Fix \(d_{0},\cdots,d_{l}\in{\C^{\times}}\). Let \(w\) be a function \(w:U\subset\C\rightarrow\mathfrak{h}\) where \(U\) is an open subset. Then the Toda equations are
\[2w_{t\overline{t}}=-\sum_{j=1}^{l}d_{j}e^{-2\alpha_{j}(w)}H_{\alpha_{j}}-d_{0}e^{2\psi(w)}H_{-\psi}.
\]
\indent If we consider the connection form \(\alpha\) 
\[\alpha=(w_{t}+\frac{1}{\lambda}\tilde{E}_{-})dt+(-w_{\bar{t}}+\lambda\tilde{E}_{+})d\bar{t}=:\alpha'dt+\alpha''d\bar{t}
\]
where \(\tilde{E}_{\pm}=Ad(e^{w})(\sum_{j=1}^{l}c_{j}^{\pm}e_{\pm\alpha_{j}}+c_{0}^{\pm}e_{\mp\psi})\) for \(c_{i}^{\pm}\in{\C^{\times}}\), then the curvature \(d\alpha+\alpha\wedge\alpha\) is zero if and only if the Toda equations hold.\\
\indent Given a real form of \(\g\), the corresponding real form of the Toda equations is defined by imposing two reality conditions: \(\alpha_{j}(w)\in{\R}\) for all \(i\), and \(\alpha'(t,\bar{t},\lambda)\mapsto\alpha''(t,\bar{t},1/\bar{\lambda})\) under the conjugation with respect to the real form.\\
\indent We add further conditions motivated by the tt* equations. Following Kostant \cite{Ko}, we introduce \(h_{0}=\sum_{j=1}^{l}\epsilon_{j}=\sum_{j=1}^{l}r_{j}H_{\alpha_{j}}\), \(e_{0}=\sum_{j=1}^{l}a_{j}e_{\alpha_{j}}\) and \(f_{0}=\sum_{j=1}^{l}(r_{j}/a_{j})e_{-\alpha_{j}}\) where \(r_{j}\in{\R^{\times}}\) and \(a_{j}\in{\C^{\times}}\). Since these generators satisfy the conditions \([h_{0},e_{0}]=e_{0}\), \([h_{0},f_{0}]=-f_{0}\) and \([e_{0},f_{0}]=h_{0}\), this subalgebra is isomorphic to \(\mathfrak{sl}_{2}\C\). We can decompose \(\g\) according to the adjoint action by this subalgebra, and then we obtain highest weight vectors \(u_{j}\) of irreducible subrepresentations \(V_{j}\) of \(\g=\bigoplus_{j}V_{j}\).\\
\indent We use the standard compact real form \(\rho\) which satisfies 
\[\rho(e_{\alpha})=-e_{-\alpha},\;\;\rho(H_{\alpha})=-H_{\alpha},
\]
for all \(\alpha\in\triangle\). By Hitchin [Hit], we have a \(\C\)-linear involution \(\sigma:\g\rightarrow \g\) defined by
\[
\sigma(u_{j})=-u_{j}, \;\;\sigma(f_{0})=-f_{0}\;\;(0\leq j\leq l).
\]
Using \(\rho\) and \(\sigma\), we define 
\[\chi:=\sigma\rho.
\]
Then it can be shown that \(\sigma\rho=\rho\sigma\) (\cite{Hit}) and that this \(\chi\) defines a split real form \(\g_{\R}\).
\begin{defi}(The tt*-Toda equations)
	The tt*-Toda equations are the Toda equations for \(w:\C^{\times}\rightarrow\mathfrak{h}\) together with\\
	(R) the above reality conditions (with respect to \(\chi\))\\
	(F) \(\sigma(w)=w\) (Frobenius condition) and\\
	(S) \(w=w(|t|)\) (similarity condition)\\
From (R) it follows that \(w\) takes values in \(\mathfrak{h}_{\sharp}=\bigoplus_{j=1}^{l}\R H_{\alpha_{j}}\).	
\end{defi}
\begin{rem}
	It is known that \(\sigma\) is the identity on \(\mathfrak{h}\) unless \(\g\) is of type \(A_{n}\), \(D_{2n+1}\) or \(E_{6}\). Thus the Frobenius condition on \(w\) is nontrivial only for these three types.
\end{rem}
\indent By the well known DPW construction (see \cite{GIL},\cite{GH}), it is possible to construct a local solution \(w\) near \(t=0\) from the connection form 
\[
\omega=\frac{1}{\lambda}\left(\sum_{j=1}^{l}z^{k_{j}}e_{-\alpha_{j}}+z^{k_{0}}e_{\psi}\right)dz
\]
(i.e. from any \(k_{0},\cdots,k_{l}\geq -1\)). Here \(z\) is a complex variable related to \(t\) by 
\[t=sz^{\frac{1}{s}}.
\] 
This solution satisfies 
\[w(|t|)\sim -m\text{log}|t|
\]
as \(t\rightarrow 0\), where \(m\in{\mathfrak{h}_{\sharp}}\) is defined by 
\[\alpha_{j}(m)=s(k_{j}+1)-1,\; 1\leq j\leq l.
\]
In fact, the converse is true.
\begin{pro}\label{pro4}{\cite{GH}}
	Let \(m\in{\mathfrak{h}_{\sharp}}\). There exists a local solution near zero of the tt*-Toda equations such that \(w(|t|)\sim -m\log{|t|}\) as \(t\rightarrow 0\) if and only if \(\alpha_{j}(m)\geq -1\) for \(j=0,\cdots,l\).
\end{pro}
The condition \(\alpha_{j}(m)\geq -1\) for \(j=0,\cdots,l\) is equivalent to the condition defining the fundamental Weyl alcove \(\mathfrak{A}=\{x\in{\sqrt{-1}\mathfrak{h}_{\sharp}}|\;\alpha_{j}^{\text{real}}(x)\geq 0,\;\psi^{\text{real}}(x)\leq 1\}\). This gives :
\begin{thm}\cite{GH}\label{thm5}
	We have a bijective map between\\
	(a) the space of asymptotic data \(\mathcal{A}=\{m\in{\mathfrak{h}_{\sharp}}|\;\alpha_{j}(m)\geq -1,\; j=0,\cdots, l\}\) when \(G\neq A_{n},D_{2m+1},E_{6}\) (or the set \(\mathcal{A}^{\sigma}=\{m\in{\mathcal{A}|\;\sigma(m)=m}\}\) when \(G=A_{n},D_{2m+1},E_{6}\)) and\\
	(b) the fundamental Weyl alcove \(\mathfrak{A}\) (or \(\mathfrak{A}^{\sigma}=\{x\in{\mathfrak{A}}|\;\sigma(x)=x\}\)) defined by
	\[\mathcal{A}\rightarrow \mathfrak{A},\;m\mapsto\frac{2\pi\sqrt{-1}}{s}(m+h_{0}),\;\;(or\;\mathcal{A}^{\sigma}\rightarrow \mathfrak{A}^{\sigma}).
	\]
\end{thm}


\subsection{Dubrovin connection}
\hspace{1pt}\\
\indent In this subsection, we review briefly the definition of (small) quantum cohomology and the corresponding Dubrovin connection. As we need only the case of compact K\"{a}hler homogeneous spaces, we can use a naive definition of Gromov-Witten invariants. We use the same notation in \cite{G2}. We set the coefficients of homology groups and cohomology rings to be \(\C\).\\
\indent Let \(M\) be such a complex manifold. Let \(p_{1},p_{2},p_{3}\) be three distinct points in \(\C P^{1}\). Let \(A,B,C\) be (generic representatives of) homology classes of \(H_{*}(M)\) and \(D\) be an element of \(H_{2}(M; \mathbb{Z})=\pi_{2}(M)\). We define 
\[\text{Hol}_{D}^{A,p_{1}}=\{\text{holomorphic maps}\;f:\C P^{1}\rightarrow M|\;f(p_{1})\in{A},\;[f]=D\}
\] 
where \([f]\) means the homotopy class of \(f\). \(\text{Hol}_{D}^{B,p_{2}},\text{Hol}_{D}^{C,p_{3}}\) are defined in the same way.
\begin{defi} Gromov-Witten invariants are defined by 
\[\langle A|B|C\rangle_{D}=\sharp \text{Hol}_{D}^{A,p_{1}}\cap\text{Hol}_{D}^{B,p_{2}}\cap\text{Hol}_{D}^{C,p_{3}}.
\]
\end{defi}
We define the quantum product for \(M\) as follows.
\begin{defi}
For \(C\in{H_{*}(M)}\) and \(\tau\in{H^{2}(M)}\), \(a\circ_{\tau}b\) is defined by
\[\langle a\circ_{\tau}b,C\rangle=\sum_{D\in{H_{2}(M)}}\langle A|B|C\rangle_{D}e^{\langle \tau,D\rangle}
\]
where \(A,B\) are the Poincar\'{e} dual homology classes to \(a,b\) and \(\langle,\rangle\) is the pairing between \(H^{*}(M)\) and \(H_{*}(M)\).
\end{defi}
We call \((H^{*}(M),\circ_{\tau})\) the quantum cohomology of \(M\) and denote it by \(QH^{*}(M)\). We denote \(\text{dim}H^{2}(M)\) by \(r\) and take the basis \(A_{1},\cdots, A_{r}\) of \(H_{2}(M)\) and the basis \(b_{1},\cdots, b_{r}\) of \(H^{2}(M)\) such that \(\langle b_{i},A_{j}\rangle=\delta_{ij}\). Let \(\tau=\tau_{1}b_{1}+\cdots+ \tau_{r}b_{r}\) and \(D=D_{1}A_{1}+\cdots+D_{r}A_{r}\) where \(\tau_{i},D_{i}\in{\C}\) for all \(i\). Then we have \(e^{\langle \tau, D\rangle}=e^{\sum_{j=1}^{r}\tau_{j}D_{j}}\).\\
\indent Finally we define the Dubrovin connection by using the quantum product \(\circ_{\tau}\).
\begin{defi}
The Dubrovin connection on the trivial vector bundle \(H^{2}(M;\C)\times H^{*}(M;\C)\rightarrow H^{2}(M;\C)\) is defined by 
\[
\nabla=d+\frac{1}{\lambda}\sum_{j=1}^{r}A_{j}(\tau)d\tau_{j}
\] 
where \(A_{i}(\tau)\) are the operators given by the quantum product \(b_{i}\circ_{\tau}\).
\end{defi}
For minuscule flag manifolds cases, we have \(r=1\). It is convenient to write the Dubrovin connection form as \(\frac{1}{\lambda}A(q)\frac{dq}{q}\) by changing the variable to \(q=e^{\tau_{1}}\). We seek flag manifolds whose Dubrovin connection forms are of the form \(\omega=\frac{1}{\lambda}(\sum_{j=1}^{l}q^{k_{j}}e_{-\alpha_{j}}+q^{k_{0}}e_{\psi})dq\).\\


\subsection{Minuscule weights and homogeneous spaces}
\hspace{1pt}\\
\indent We review some properties of minuscule weights. We refer to the article \cite{CMP}. For a simple complex Lie algebra, we define the weight lattice \(I\) as the \(\mathbb{Z}\)-module spanned by \(\lambda_{1},\cdots,\lambda_{l}\) where \(\lambda_{i}\) is defined by \((\lambda_{i},\alpha^{\vee}_{j})=\delta_{ij}\). These \(\lambda_{i}\) are called the fundamental weights.
\begin{defi}
We call a weight \(\lambda\) a dominant weight if \((\lambda,\alpha^{\vee}_{i})>0\) for all \(\alpha_{i}\in{\Pi}\). We call a dominant weight \(\lambda\) a minuscule weight if \((\lambda,\alpha^{\vee})\leq 1\) for all \(\alpha\in{\triangle^{+}}\).
\end{defi}
It is well-known that the minuscule weights are a subset of the fundamental weights. We summarize the minuscule weights for each types of Lie groups at the end of this subsection.\\
\indent By the Borel-Weil theorem, we can obtain an irreducible representation \(V_{\lambda_{i}}\) from each fundamental weight \(\lambda_{i}\). When we consider the projective representation \(\mathbb{P}(V_{\lambda_{i}})\), we obtain the homogeneous space 
\[G/P_{i}\cong G\cdot [v_{\lambda_{i}}]\subset{\mathbb{P}(V_{\lambda_{i}})}
\]
where \(v_{\lambda_{i}}\) is a highest weight vector and \(P_{i}\) is the stabilizer group of \([v_{\lambda_{i}}]\). Here \(P_{i}\) is a parabolic subgroup.\\
\indent We denote the weight orbit of \(\lambda_{i}\) by \(W\cdot \lambda_{i}\). That is \(W\cdot \lambda_{i}=\{x(\lambda_{i})|\;x\in{W}\}\). When we write \(x\) as a product of simple reflections, we denote by \(\ell(x)\) the minimal length of \(x\) in \(W\). The following fact holds for any parabolic subgroup \(P\) of \(G\). Let \(\triangle_{P}\) be the subset of \(\triangle\) such that Lie\((P)=\mathfrak{h}\oplus \bigoplus_{\alpha\in{\triangle_{P}}}\g_{\alpha}\) We denote the subset of the simple roots which belong to \(\triangle_{P}\) by \(\Pi_{P}\). Let \(W_{P}\) be the subgroup of \(W\) generated by the elements of \(\Pi_{P}\).
\begin{pro}(see section 1.10 in \cite{Hum})\label{pro9}
	For \(x\in{W}\), there exist unique elements \(u\in{W^{P}}\) and \(v\in{W_{P}}\) such that 
	\[
	x=uv
	\] 
	where \(W^{P}=\{x\in{W}|\; \ell(xs_{\alpha})>\ell(x)\;{}^{\forall}\alpha\in{\Pi_{P}}\}\).
\end{pro}
\noindent By this fact, \(u\) is a representative of \([x]\in{W/W_{P}}\). We have \(W\cdot \lambda_{i}=W^{P_{i}}\cdot \lambda_{i}\).\\
\indent We shall consider the cohomology ring of \(G/P_{i}\). The following fact is well-known.
\begin{thm}(Bruhat decomposition)\cite{Hil}
	For a parabolic subgroup \(P\) of \(G\), we have a decomposition 
	\[G=\coprod_{u\in{W^{P}}}BuP.
	\]
\end{thm}
We define the Schubert varieties of \(G/P_{i}\) by \(X_{u}:=\overline{BuP_{i}/P_{i}}\). We also define the opposite Schubert varieties by \(Y_{u}:=\overline{x_{0}Bx_{0}uP_{i}/P_{i}}=x_{0}X_{x_{0}u}\) where \(x_{0}\) is the longest element of \(W\). Then \([Y_{u}]\in{H_{2n-2l(u)}(G/P_{i})}\) and these classes form an additive basis. By the Poincar\'{e} duality theorem, we have a basis of \(H^{2l(u)}(G/P_{i})\). We denote this generator by \(\sigma_{u}\).\\
\indent Now we obtain the correspondence between \(W^{P_{i}}\cdot \lambda_{i}\) and an additive basis of the cohomology \(H^{*}(G/P_{i})\) by 
\[u(\lambda_{i})\longleftrightarrow \sigma_{u}.
\]
In the following table of fundamental weights (figure 1), the minuscule weights are marked.
\begin{figure}[t]
\begin{center}
\includegraphics{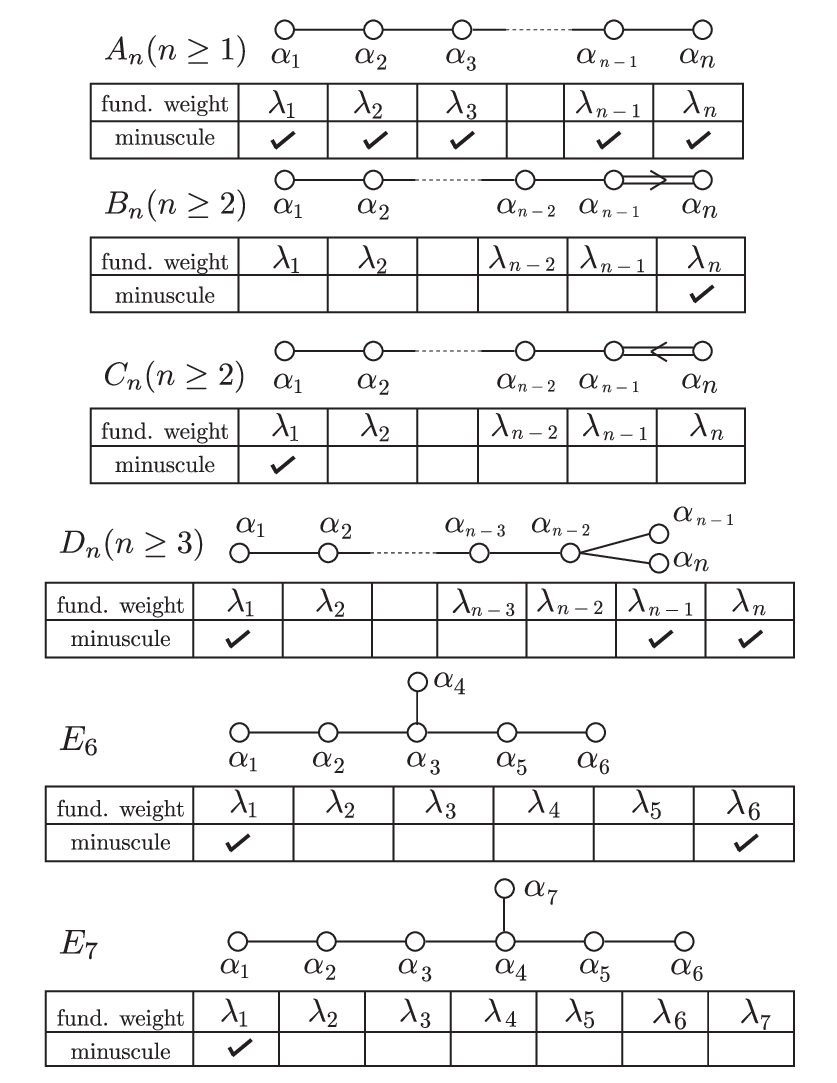}
\end{center}
\caption{fundamental weight}
\end{figure}
It is known that \(G_{2},F_{4}\) and \(E_{8}\) have no minuscule weight. \(G/P_{\lambda_{i}}\) can be described conveniently as a quotient of compact groups as follows.
\begin{align*}
(A_{n}\; \text{case})\;\; &G/P_{i}\cong SU(n+1)/S(U(i)\times U(n+1-i))\cong Gr(k,n+1)\\
(B_{n}\; \text{case})\;\; &G/P_{n}\cong SO(2n+1)/U(n) \cong OG(n,2n+1)\\
(C_{n}\; \text{case})\;\; &G/P_{1}\cong Sp(n)/U(1)\times Sp(n-1)\cong \C P^{2n-1}\\
(D_{n}\; \text{case})\;\; &G/P_{1}\cong SO(2n)/U(1)\times SO(2n-2) \cong Q_{2n-2},\\
&G/P_{n-1}\cong SO(2n)/U(n)\cong S_{+},G/P_{n}\cong SO(2n)/U(n)\cong S_{-}\\
(E_{6}\; \text{case})\;\; &G/P_{1}\cong G/P_{6}\cong E_{6}/SO(10)\times U(1)\cong \mathbb{O}P^{2}\\
(E_{7}\; \text{case})\;\; &G/P_{1}\cong E_{7}/E_{6}\times U(1)
\end{align*}
Here \(OG(k,n)\) is the set of \(k\)-dimensional isotropic subspaces of \(n\)-dimensional complex vector space \(V\) with a nondegenerate quadratic form. This is called the orthogonal Grassmannian. For \(D_{n}\), \(OG(n,2n)\) has two components \(S_{+}\) and \(S_{-}\). These are called varieties of pure spinors (or spinor varieties) and these are isomorphic to each other \cite{Ma}. For \(A_{n},B_{n},C_{n}\) and \(D_{n}\), the minuscule representations are familiar (see section 6.5 in \cite{BD}). For \(A_{n}\), \(V_{\lambda_{i}}\) is the exterior power \(\bigwedge^{i}V_{\lambda_{1}}\) (\(1\leq i\leq n\)) where \(V_{\lambda_{1}}\) is the standard representation on \(\C^{n+1}\). For \(B_{n}\), \(V_{n}\) is the half-spin representation. For \(C_{n}\), \(V_{\lambda_{1}}\) is the standard representation on \(\C^{2n}\). For \(D_{n}\), \(V_{\lambda_{1}}\) is the standard representation on \(\C^{2n}\). \(V_{\lambda_{n-1}}\) and \(V_{\lambda_{n}}\) are the half-spin representations. We denote these two representations by \(\Delta_{+}\) and \(\Delta_{-}\). For exceptional groups, the minuscule representations are given in the section 5 of \cite{Gec}. For \(E_{6}\), \(V_{\lambda_{1}}\) and \(V_{\lambda_{6}}\) are \(27\) dimensional representations. For \(E_{7}\), \(V_{\lambda_{1}}\) is a \(56\) dimensional  representation.
\\

\subsection{Minuscule weight orbits and simple roots}
\hspace{1pt}\\
\indent In this subsection, we observe relationships between minuscule weight orbits and the simple roots. Let \(\lambda_{i}\) be a minuscule weight.
\begin{pro}\label{pro6}
The set of all weights of \(V_{\lambda_{i}}\) is the \(W\)-orbit of \(\lambda_{i}\) and the multiplicities of all weights of \(V_{\lambda_{i}}\) are one.
\end{pro}
\begin{proof}
\indent It is obvious that \(\sharp W/W_{P_{i}}\leq \text{dim}(W\cdot v_{\lambda_{i}})\leq \text{dim}(V_{\lambda_{i}})\). If there is a weight which has multiplicity more than one, then \(\sharp W/W_{P_{i}}< \text{dim}V_{\lambda_{i}}\). Therefore by contraposition when we show that \(\sharp W/W_{P_{i}}\) coincides with \(\text{dim}_{\C}V_{\lambda_{i}}\), we obtain the statement of proposition \ref{pro6}.\\
\indent We justify the above claim in each case. We have the orders of all Weyl groups from the table 2 in section 2.11 of \cite{Hum}. For type \(A_{n}\), we have \(\text{dim}_{\C}\bigwedge {}^{i}\C^{n+1}=\binom{n+1}{i}\) (\(1\leq i\leq n\)). On the other hand, for this representation we have \(W/W_{P_{i}}=\mathfrak{S}_{n+1}/(\mathfrak{S}_{i}\times \mathfrak{S}_{n+1-i})\). Therefore we obtain \(\sharp W/W_{P_{i}}=\frac{(n+1)!}{i!(n+1-i)!}=\binom{n+1}{i}\). For type \(B_{n}\), a minuscule representation is the half-spin representation and its dimension is \(2^{n}\). Then \(W/W_{P_{n}}=\mathfrak{S}_{n}\cdot (\mathbb{Z}_{2})^{n}/\mathfrak{S}_{n}\). Hence \(\sharp W/W_{P_{n}}=2^{n}\cdot n!/n!=2^{n}\). For type \(C_{n}\), a minuscule representation is the standard representation \(\C^{2n}\) and its dimension is \(2n\). The corresponding \(W/W_{P_{1}}=\mathfrak{S}_{n}\cdot (\mathbb{Z}_{2})^{n}/\mathfrak{S}_{n-1}\cdot (\mathbb{Z}_{2})^{n-1}\). Hence \(\sharp W/W_{P_{1}}=2^{n}\cdot n!/2^{n-1}\cdot (n-1)!=2n\). For type \(D_{n}\), there are three minuscule representations. These are the standard representations and the two half-spin representations. These dimensions are \(2n\), \(2^{n-1}\), \(2^{n-1}\) respectively. The corresponding \(W/W_{P_{i}}\) (\(i=1,n-1,n\)) are \(\mathfrak{S}_{n}\cdot (\mathbb{Z}_{2})^{n-1}/\mathfrak{S}_{n-1}\cdot (\mathbb{Z}_{2})^{n-2}\), \(\mathfrak{S}_{n}\cdot (\mathbb{Z}_{2})^{n-1}/\mathfrak{S}_{n}\), \(\mathfrak{S}_{n}\cdot (\mathbb{Z}_{2})^{n-1}/\mathfrak{S}_{n}\), and \(\sharp W/W_{P_{i}}\)(\(i=1,n-1,n\)) are \(2n\),  \(2^{n-1}\), \(2^{n-1}\) respectively. For type \(E_{6}\), there are two minuscule representations. These representations are both \(27\) dimensional representations. The corresponding \(W/W_{P_{1}}\) and \(W/W_{P_{6}}\) are both \(W_{E_{6}}/\mathfrak{S}_{5}\cdot (\mathbb{Z}_{2})^{4}\) where \(W_{E_{6}}\) is the Weyl group of \(E_{6}\). Then \(\sharp W_{E_{6}}/\mathfrak{S}_{5}\cdot (\mathbb{Z}_{2})^{4}=2^{7}\cdot 3^{4}\cdot 5/2^{4}\cdot 5!=27\). For type \(E_{7}\), the minuscule representation is a \(56\) dimensional representation. The corresponding \(W/W_{P_{1}}\) is \(W_{E_{7}}/W_{E_{6}}\) where \(W_{E_{7}}\) is the Weyl group of \(E_{7}\). Then \(\sharp W/W_{P_{1}}=2^{10}\cdot 3^{4}\cdot 5\cdot 7/2^{7}\cdot 3^{4}\cdot 5=56\). This completes the proof. 
\end{proof}
From proposition \ref{pro6}, we have the weights of \(V_{\lambda_{i}}\) as \(\{v_{u(\lambda_{i})}|\; u\in{W^{P_{i}}}\}\) and the multiplication of these weights are all one. In addition, we know that the Weyl group is generated by the simple reflections \(\{s_{\alpha_{j}}|\;j\in{\{1,\cdots,l\}}\}\). Therefore all weights can be obtained from  \(\lambda_{i}\) by applying \(\{s_{\alpha_{j}}|\; j\in{\{1,\cdots,l\}}\}\) to \(\lambda_{i}\) repeatedly. We use a canonical basis of \(V_{\lambda_{i}}\) from section \(5A.1\) of the article \cite{Ja} with the following properties:
\begin{equation}
e_{-\alpha_{j}}(v_{u(\lambda_{i})})=
\begin{cases}\label{eq2.2}
v_{u(\lambda_{i})-\alpha_{j}} &(u(\lambda_{i}),\alpha_{j}^{\vee})=1\\
0 & \text{otherwise}.
\end{cases}
\end{equation} 
\begin{equation}
e_{\alpha_{j}}(v_{u(\lambda_{i})})=
\begin{cases}\label{eq2.3}
v_{u(\lambda_{i})+\alpha_{j}} &(u(\lambda_{i}),\alpha_{j}^{\vee})=-1\\
0 & \text{\rm otherwise}.
\end{cases}
\end{equation}
\[
H_{\alpha_{j}}(v_{u(\lambda_{i})})=(u(\lambda_{i}),\alpha_{j}^{\vee})v_{u(\lambda_{i})},
\]
for all weights \(u(\lambda_{i})\) and all \(j\in{\{1,\cdots,l\}}\). As a consequence of (\ref{eq2.3}), we have
\begin{equation}\label{eq2.4}
e_{\psi}(v_{u(\lambda_{i})})=
\begin{cases}
	v_{u(\lambda_{i})+\psi}&(u(\lambda_{i}),\psi^{\vee})=-1\\
	0 & \text{otherwise}.
\end{cases}
\end{equation}

\subsection{Results}
\hspace{1pt}\\
\indent For any minuscule weight \(\lambda_{i}\), the discussion in 2.3 establishes an isomorphism
\[V_{\lambda_{i}}= \bigoplus_{u(\lambda_{i})\in{W^{P}}\cdot \lambda_{i}}V_{u(\lambda_{i})}\cong H^{*}(G/P_{i};\C).
\]
We remark that from section 2.3 this isomorphism is given by 
\[V_{u(\lambda_{i})}\ni v_{u(\lambda_{i})}\leftrightarrow \sigma_{u}\in{H^{2\ell(u)}(G/P_{i});\C}
\]
for all \(u\in{W^{P_{i}}}\). From this it can be seen that the cohomology grading on the right corresponds to the grading by simple roots on the left.\\
\indent Now we can state our main theorem.
\begin{thm}\label{main}
Fix \(\g\) and a minuscule weight \(\lambda_{i}\). There is a natural correspondence between (i) the asymptotic data 
\[
m=-h_{0}=-\sum_{j=1}^{l}r_{j}H_{\alpha_{j}}\in{\mathfrak{h}_{\sharp}}
\]
and (ii) the DPW data 
\[\omega =\frac{1}{\lambda}\left(\sum_{j=1}^{l}e_{-\alpha_{j}}+qe_{\psi}\right)\frac{dq}{q}
\]
for solutions of the tt*-Toda equations. The asymptotic data corresponds to a unique global solution when \(\g\) has type \(A_{n}\) (and conjecturally for any \(\g\)). The holomorphic data correspond to the Dubrovin connection for the quantum cohomology of \(G/P_{i}\), i.e. the natural action of \(\sum_{j=1}^{l} e_{-\alpha_{j}}+qe_{\psi}\) corresponds to quantum multiplication by a generator of \(H^{2}(G/P_{i},\C)\).
\end{thm}
\begin{proof}
	In the bijection of Theorem \ref{thm5} (section 2.1), we see that \(m=-h_{0}\) corresponds to the origin of the fundamental Weyl alcove, and in this case we have \(k_{0}=0\) and \(k_{1}=\cdots=k_{l}=-1\). This gives the correspondence between (i) and (ii) (with \(q=z\)). For the statement concerning global solutions, we refer to \cite{GIL}, \cite{Mo}. The identification of \(\omega\) with the Dubrovin connection can be extracted from \cite{GM}, but we shall present a new\footnote{After finishing the first draft of this paper we found essentially the same proof is given in \cite{LT}.} and more direct proof in the next section.
\end{proof}
\begin{rem} (On the Satake isomorphism) \rm{When \(\g\) is of type \(A_{n}\) (or, conjecturally, of type \(D_{n}\),\(E_{6}\)), the same global solution corresponds to the Dubrovin connection of any minuscule weight. This suggests a relation between the quantum cohomology algebra of the corresponding flag manifolds. In the \(A_{n}\) case this can be stated as 
\[\textstyle\bigwedge^{k}QH^{*}(\C P^{n})\cong QH^{*}Gr(k,n+1)
\]
(see \cite{GM} for further explanation).\\
\indent In the \(D_{n}\) case, the analogous relation is: 
\begin{align}\label{rem14}
\textstyle\bigwedge_{\pm}^{half}QH^{*}(Q_{2n-2})\cong \text{End}_{\C}(QH^{*}(S_{\pm})).
\end{align}
This follows from theorem \ref{main} when we identify \(H^{*}(Q_{2n-2};\C)\) with \(\C^{2n}\) and \(H^{*}(S_{\pm};\C)\) with \(\Delta_{\pm}\), because (\ref{rem14}) corresponds to the well known relation
\[\textstyle\bigwedge_{\pm}^{half}\C^{2n}\cong \text{End}_{\C}(\Delta_{\pm}).
\]
In order to explain the notation, we recall the relation here. We denote a positively oriented orthonormal basis of \(\C^{2n}\) by \(e_{1},\cdots, e_{2n}\). We define the isomorphism \(\star:\textstyle\bigwedge^{i}\C^{2n}\rightarrow \textstyle\bigwedge^{2n-i}\C^{2n}\) by 
\[\star(e_{\xi(1)}\wedge e_{\xi(2)}\wedge \cdots\wedge e_{\xi(i)})=\text{sign}(\xi)e_{\xi(i+1)}\wedge e_{\xi(i+2)}\wedge \cdots \wedge e_{\xi(2n)}\]
for any permutation \(\xi\). Then we obtain \(\star\cdot \star =(-1)^{i(2n-i)}\text{id}\). We define \(\iota:=(-i)^{n}\star: \bigwedge^{n}\C^{2n}\rightarrow \bigwedge^{n}\C^{2n}\). Then \(\iota\cdot\iota=\text{id}\). Thus we have the canonical eigenspace decomposition \(\bigwedge^{n}\C^{2n}\cong \bigwedge_{+}^{n}\C^{2n}\oplus \bigwedge_{-}^{n}\C^{2n}\). If \(n=2m+1\), then we define \(\bigwedge_{\pm}^{half}\C^{2n}\) by
\[\textstyle\bigwedge^{0}\C^{4m+2}\bigoplus \textstyle\bigwedge^{2}\C^{4m+2}\bigoplus \cdots \bigoplus \textstyle\bigwedge^{2m}\C^{4m+2}.
\]
If \(n=2m\), then we define \(\bigwedge_{+}^{half}\C^{2n}\) by
\[\textstyle\bigwedge^{0}\C^{4m}\bigoplus \textstyle\bigwedge^{2}\C^{4m}\bigoplus \cdots \bigoplus \textstyle\bigwedge_{+}^{2m}\C^{4m}.
\]
and \(\bigwedge_{-}^{half}\C^{2n}\) by
\[\textstyle\bigwedge^{0}\C^{4m}\bigoplus \textstyle\bigwedge^{2}\C^{4m}\bigoplus \cdots \bigoplus \textstyle\bigwedge_{-}^{2m}\C^{4m}.
\]
\noindent From Theorem (6.2) of \cite{BD}, we have 
\begin{align*}
	& \Delta_{+}\otimes \Delta_{+}=\textstyle\bigwedge_{+}^{n}+\textstyle\bigwedge^{n-2}+\cdots \\
	& \Delta_{+}\otimes \Delta_{-}=\textstyle\bigwedge^{n-1}+\textstyle\bigwedge^{n-3}+\cdots \\
	& \Delta_{-}\otimes \Delta_{-}=\textstyle\bigwedge_{-}^{n}+\textstyle\bigwedge^{n-2}+\cdots 
\end{align*}
as \(\mathfrak{spin}(2n)\) representations where the last terms are \(\bigwedge^{4}+\bigwedge^{2}+\bigwedge^{0}\) or \(\bigwedge^{3}+\bigwedge^{1}\). If \(n=2m+1\), then we have 
\begin{align*}
\text{End}_{\C}(\Delta_{+})&\cong \Delta_{+}^{*}\otimes \Delta_{+}\cong \Delta_{+}\otimes \Delta_{-}\\
& \cong \textstyle\bigwedge^{2m}+\textstyle\bigwedge^{2m-2}+\cdots+\textstyle\bigwedge^{2}+\textstyle\bigwedge^{0}\\
& =\textstyle\bigwedge_{\pm}^{half}\C^{4m+2}. 
\end{align*}
If \(n=2m\), then we have 
\begin{align*}
\text{End}_{\C}(\Delta_{+})&\cong \Delta_{+}^{*}\otimes \Delta_{+}\cong \Delta_{+}\otimes \Delta_{+}\\
& \cong \textstyle \bigwedge_{+}^{2m}+\bigwedge^{2m-2}+\cdots+\bigwedge^{2}+\bigwedge^{0}\\
& =\textstyle\bigwedge_{+}^{half}\C^{4m}. 
\end{align*}
When we consider the minuscule \(\Delta_{-}\) and the corresponding homogeneous space \(S_{-}\), we obtain 
\[\textstyle\bigwedge_{-}^{half}QH^{*}(Q_{2n-2})\cong \text{End}_{\C}(QH^{*}(S_{-}))
\]
as in the case of \(\Delta_{+}\).
}
\end{rem}

\section{Completion of the proof of the main theorem}
\hspace{1pt}\\
\indent We consider the irreducible representations \(V_{\lambda_{i}}\) whose highest weight are minuscule weights \(\lambda_{i}\) (see table in section 2.3). In this section we use results on quantum cohomology to prove that the quantum multiplication by the generator of the second cohomology coincides with the endomorphism \(\sum_{j=1}^{l}e_{-\alpha_{j}}+qe_{\alpha_{\psi}}\) for a minuscule representation \(V_{\lambda_{i}}\). To show this statement, we use the quantum Chevalley formula.
\begin{thm}{(\cite{FW})} For \(\beta\in{\Pi\backslash\Pi_{P_{i}}}\) and \(u\in{W^{P_{i}}}\), we have the quantum product \(\circ\) by \(\sigma_{\beta}\) as
\begin{align*}
	\sigma_{s_{\beta}}\circ \sigma_{u}=&\sum_{\ell(us_{\alpha})=\ell(u)+1}(\lambda_{\beta},\alpha^{\vee})\sigma_{us_{\alpha}}\\
	&+\sum_{l(us_{\alpha})=l(u)-n_{\alpha}+1}(\lambda_{\beta},\alpha^{\vee})\sigma_{us_{\alpha}}\cdot q^{d(\alpha)}
\end{align*}
where \(\alpha\) ranges over \(\triangle^{+}\backslash \triangle_{P_{i}}^{+}\), \(\lambda_{\beta}\) is the fundamental weight corresponding to \(\beta\), 
\[
n_{\alpha}=(\sum_{\gamma\in{\triangle^{+}\backslash\triangle_{P_{i}}^{+}}} \gamma,\alpha^{\vee})
\]
and
\[
d(\alpha)=\sum_{\beta\in{\Pi\backslash\Pi_{P_{i}}}}(\lambda_{\beta},\alpha^{\vee})\sigma(s_{\beta}),
\]
and where \(\sigma(s_{\beta})\) is the homology class of \(H_{2}(G/P_{i})\) which is Poincar\'e dual to \(\sigma_{s_{\beta}}\).
\end{thm}
In our situation, \(\Pi\backslash\Pi_{P_{i}}=\{\alpha_{i}\}\). Therefore the generator of the second cohomology is only \(\sigma_{s_{\alpha_{i}}}\) and \(\lambda_{\beta}=\lambda_{i}\). We have \(d(\alpha)=(\lambda_{i},\alpha^{\vee})\sigma(s_{\alpha_{i}})=\sigma(s_{\alpha_{i}})\) for \(\alpha\in{\triangle^{+}\backslash\triangle^{+}_{P_{i}}}\) because \(\lambda_{i}\) is a minuscule weight. We consider \(q^{\sigma(s_{\beta})}\)  only as a complex parameter \(q\) in \(\C\).\\ 
\indent From lemma 3.5 in \cite{FW}, the first Chern class of \(G/P_{i}\) is \(n_{\alpha}\) times a generator of \(H^{2}(G/P_{i})\), and by \cite{CMP}, we know that \(n_{\alpha}\) is the Coxeter number \(s\). Explicitly, we have \(n_{\alpha}=n+1\) (\(A_{n}\) type), \(n_{\alpha}=2n\) (\(B_{n}\) type), \(n_{\alpha}=2n\) (\(C_{n}\) type), \(n_{\alpha}=2n-2\) (\(D_{n}\) type), \(n_{\alpha}=12\) (\(E_{6}\) type), \(n_{\alpha}=18\) (\(E_{7}\) type) for all \(\alpha\in{\triangle^{+}\backslash \triangle_{P_{i}}^{+}}\).\\
Then we have the quantum Chevalley formula as follows.
\begin{align*}
	\sigma_{s_{\alpha_{i}}}\circ \sigma_{u}=&\sum_{\ell(us_{\alpha})=\ell(u)+1}(\lambda_{i},\alpha^{\vee})\sigma_{us_{\alpha}}\\
	&+\sum_{\ell(us_{\alpha})=\ell(u)-(s-1)}(\lambda_{i},\alpha^{\vee})\sigma_{us_{\alpha}}\cdot q
\end{align*}
where \(\alpha\in{\triangle^{+}\backslash\triangle_{P_{i}}^{+}}\).\\
\indent To replace the conditions of these summations, the following lemma, corollary and proposition are key ingredients.
\begin{lem}\label{key} Let \(\lambda_{i}\) be a minuscule weight. For \(u\in{W^{P_{i}}}\) and \(\alpha\in{\Pi}\), we have the three following situations.\\
	(I) \((u(\lambda_{i}),\alpha^{\vee})=1\Leftrightarrow \ell(s_{\alpha}u)=\ell(u)+1\).\\
	(II) \((u(\lambda_{i}),\alpha^{\vee})=0\Leftrightarrow \ell(s_{\alpha}u)=\ell(u)\).\\
	(III) \((u(\lambda_{i}),\alpha^{\vee})=-1\Leftrightarrow \ell(s_{\alpha}u)=\ell(u)-1\).\\
	Here we consider the length function \(l(u)\) in \(W^{P_{i}}\).
\end{lem}
\begin{proof}
(a) First we show the implication (\(\Rightarrow\)), for each of (I), (II), (III). Here we do not use the minuscule condition.\\
(I) We assume \((u(\lambda_{i}),\alpha^{\vee})=1\). We show \(s_{\alpha}u\in{W^{P_{i}}}\). If \((u(\lambda_{i}),\alpha^{\vee})=1\), \((\lambda_{i},u^{-1}(\alpha)^{\vee})=1\) and \(u^{-1}(\alpha)\) is a positive root. Therefore \(\ell(s_{\alpha}u)=\ell(u)+1\) in \(W\) (see section 1.6 in \cite{Hum}). For \(\beta\in{\Pi_{P_{i}}}\), we have 
\begin{align*}
	(\lambda_{i},(us_{\beta})^{-1}(\alpha)^{\vee})>0 &\Leftrightarrow (\lambda_{i},s_{\beta}u^{-1}(\alpha)^{\vee})>0\\
	&\Leftrightarrow (s_{\beta}(\lambda_{i}),u^{-1}(\alpha)^{\vee})>0\\
	&\Leftrightarrow (\lambda_{i},u^{-1}(\alpha)^{\vee})>0
\end{align*}
Therefore we have \((\lambda_{i},(us_{\beta})^{-1}(\alpha)^{\vee})>0\). Hence \(\ell(s_{\alpha}us_{\beta})>\ell(us_{\beta})\). On the other hand, for all \(\beta\in{\Pi_{P_{i}}}\), we have \(\ell(us_{\beta})>\ell(u)\) because \(u\) is in \(W^{P_{i}}\). Hence \(\ell(us_{\beta})=\ell(u)+1\) in \(W\). Thus we have
\[\ell(s_{\alpha}u)=\ell(us_{\beta})<\ell(s_{\alpha}us_{\beta}).
\]
This means that \(s_{\alpha}u\in{W^{P_{i}}}\). Therefore we obtain \(\ell(s_{\alpha}u)=\ell(u)+1\) in \(W^{P_{i}}\).\\
(II) We assume \((u(\lambda_{i}),\alpha^{\vee})=0\). We show \(s_{u^{-1}(\alpha)}\in{W_{P_{i}}}\). Let \(u^{-1}(\alpha)^{\vee}=b_{1}\alpha_{1}^{\vee}+\cdots+b_{l}\alpha_{l}^{\vee}\) (\(b_{i}\in{\R}\)). Then we have
\[(\lambda_{i},u^{-1}(\alpha)^{\vee})=b_{i}=0
\]
Therefore \(u^{-1}(\alpha)\in{\triangle_{P_{i}}}\) and \(s_{u^{-1}(\alpha)}\in{W_{P_{i}}}\). We obtain
\[\ell(s_{\alpha}u)=\ell(us_{u^{-1}(\alpha)})=\ell(u) \;\text{in}\; W^{P_{i}}.
\]
(III) We assume \((u(\lambda_{i}),\alpha^{\vee})=-1\). We show \(s_{\alpha}u\in{W^{P_{i}}}\). If \((u(\lambda_{i}),\alpha^{\vee})=-1\), then \((\lambda_{i},u^{-1}(\alpha)^{\vee})=-1<0\). \(u^{-1}(\alpha)\) is a negative root. Hence we have
\[
\ell(s_{\alpha}u)=\ell(u)-1<\ell(u)<\ell(us_{\beta}) \;\text{in}\; W
\] 
for \(\beta\in{\Pi_{P_{i}}}\). Now we have 
\[
\ell(us_{\beta})=\ell(u)+1=\ell(s_{\alpha}u)+2 \;\text{in}\; W.
\]
Let \(\ell(us_{\beta})=r\). Then \(\ell(s_{\alpha}us_{\beta})=r-1,r+1\) and \(\ell(s_{\alpha}u)=r-2\). So \(\ell(s_{\alpha}us_{\beta})>\ell(s_{\alpha}u)\). This means that \(s_{\alpha}u\in{W^{P_{i}}}\). Thus we obtain \(\ell(s_{\alpha}u)=\ell(u)-1\) in \(W^{P_{i}}\).\\
(b) Next we show the implication (\(\Leftarrow\)), for each of (I), (II), (III). For (I), we assume \(\ell(s_{\alpha u})=\ell(u)+1\). Since \(\lambda_{i}\) is minuscule, \((u(\lambda_{i}),\alpha^{\vee})\) takes only the values \(1,0,-1\). If \((u(\lambda_{i}),\alpha^{\vee})\) is \(0\) or \(-1\), we obtain a contradiction, by part (a). The proofs in the case (II), (III) are similar.  
\end{proof}
\indent Now we have the weights of \(V_{\lambda_{i}}\) as \(\lambda_{i}-\sum_{j=1}^{l}n_{j}\alpha_{j}\) where \(n_{j}\in{\Z_{\geq 0}}\). From this lemma, we obtain the following corollary.
\begin{cor}\label{cor16}
	 For \(u\in{W^{P_{i}}}\) such that \(u(\lambda_{i})=\lambda_{i}-\sum_{j=1}^{l}n_{j}\alpha_{j}\), we have \(\ell(u)=\sum_{j=1}^{l}n_{j}\).
\end{cor}
\begin{proof}
We have 
\begin{align*}
	\ell(s_{\alpha_{j}}u)=\ell(u)+1 &\Leftrightarrow (u(\lambda_{i}),\alpha_{j}^{\vee})=1\\
	&\Leftrightarrow s_{\alpha_{j}}(u(\lambda_{i}))=u(\lambda_{i})-\alpha_{j}
\end{align*}
by lemma \ref{key}. The elements of \(W^{P_{i}}\) are described by a product of simple reflections. Thus \(\ell(u)=\sum_{j=1}^{l}n_{j}\).
\end{proof}
We have the following proposition.
\begin{pro}\label{prop15}
(I) If there exist \(\alpha\in{\triangle^{+}}\) such that \(\ell(s_{\alpha}u)=\ell(u)+1\) for \(u\in{W^{P_{i}}}\), then \(\alpha\in{\Pi}\) and \((u(\lambda_{i}),\alpha^{\vee})=1\).\\
(II) If there exist \(\alpha\in{\triangle^{+}}\) such that \(\ell(s_{\alpha}u)=\ell(u)-(s-1)\) for \(u\in{W^{P_{i}}}\), then \(\alpha=\psi\) and \((u(\lambda_{i}),\psi^{\vee})=-1\).
\end{pro}
\begin{proof}
	(I) For \(\alpha\in{\triangle^{+}}\) such that \(\ell(s_{\alpha}u)=\ell(u)+1\), we have 
\[s_{\alpha}u(\lambda_{i})=u(\lambda_{i})-(u(\lambda_{i}),\alpha^{\vee})\alpha.
\]
By the assumption that \(\ell(s_{\alpha}u)>\ell(u)\), we have \((u(\lambda_{i}),\alpha^{\vee})=1\) and \(\alpha\) must be a simple root by corollary \ref{cor16}.\\
(II) For \(\alpha\in{\triangle^{+}}\) such that \(\ell(s_{\alpha}u)=\ell(u)-(s-1)\). Then we have 
\[s_{\alpha}u(\lambda_{i})=u(\lambda_{i})-(u(\lambda_{i}),\alpha^{\vee})\alpha.
\]
By the assumption \(\ell(s_{\alpha}u)<\ell(u)\), we have \((u(\lambda_{i}),\alpha^{\vee})=-1\). When \(\alpha=\sum_{j=1}^{l}q_{j}\alpha_{j}\), then \(\alpha\) must be \(\psi\) because there is only one positive root which has the height \(\sum_{j=1}^{l}q_{j}=s-1\).
\end{proof}

By using the relation \(us_{\alpha}=s_{u(\alpha)}u=s_{-u(\alpha)}u\),  corollary \ref{cor16} and proposition \ref{prop15}, we can replace the conditions of the summation in the quantum Chevalley formula.\\
\indent We show that we can simplify the first summation to 
\[\sum_{(u(\lambda_{i}),\alpha'^{\vee})=1,\alpha'\in{\Pi}}\sigma_{s_{\alpha'}u}
\] 
by setting \(\alpha'=u(\alpha)\). Then we shall show that \(\alpha'\) is a positive root. In fact, if \(\alpha'\) is a negative root, then \((u(\lambda_{i}),\alpha'^{\vee})=-1\) satisfies \(\ell(s_{\alpha'}u)=\ell(u)+1\). However this contradicts \(\alpha\in{\triangle^{+}\backslash\triangle_{P_{i}}^{+}}\) because we have 
\[(u(\lambda_{i}),\alpha'^{\vee})=-1\Leftrightarrow (u(\lambda_{i}),u(\alpha^{\vee}))=-1\Leftrightarrow (\lambda_{i},\alpha^{\vee})=-1.
\]
Thus \(\alpha'\) is in \(\triangle^{+}\). By proposition \ref{prop15}, we have \(\alpha'\in{\Pi}\subset{\triangle^{+}}\). Hence we have 
\[\sum_{\ell(us_{\alpha})=\ell(u)+1}(\lambda_{i},\alpha^{\vee})\sigma_{us_{\alpha}}=\sum_{(u(\lambda_{i}),\alpha'^{\vee})=1,\alpha'\in{\Pi}}\sigma_{s_{\alpha'}u}
\]
as the first summation of \(\sigma_{s_{\alpha_{i}}}\circ \sigma_{u}\).\\
\indent For the second summation, let \(\alpha'=-u(\alpha)\). Then we shall show that \(\alpha'\) is also a positive root. In fact, if \(\alpha'\) is a negative root, then \((u(\lambda_{i}),\alpha'^{\vee})=1\) satisfies \(\ell(s_{\alpha'}u)=\ell(u)-(s-1)\). However this contradicts \(\alpha\in{\triangle^{+}\backslash\triangle_{P_{i}}^{+}}\) because we have 
\[
(u(\lambda_{i}),\alpha'^{\vee})=1\Leftrightarrow (u(\lambda_{i}),-u(\alpha^{\vee}))=1\Leftrightarrow (\lambda_{i},\alpha^{\vee})=-1.
\]
Thus \(\alpha'=-u(\alpha)\) is in \(\triangle^{+}\) for \(\alpha\in{\triangle^{+}\backslash\triangle_{P_{i}}^{+}}\). By proposition \ref{prop15}, we have \(\alpha'=\psi\) and \((u(\lambda_{i}),\psi^{\vee})=-1\). Hence for the second summation of \(\sigma_{s_{\alpha_{i}}}\circ \sigma_{u}\) we have
\begin{align*}
&\sum_{\ell(us_{\alpha})=\ell(u)-(s-1)}(\lambda_{i},\alpha^{\vee})\sigma_{us_{\alpha}}\cdot q\\
=& \sum_{\ell(s_{\alpha'}u)=\ell(u)-(s-1)}(\lambda_{i},-u^{-1}(\alpha'^{\vee}))\sigma_{s_{\alpha'}u}\cdot q\\
=&
\begin{cases}
q\sigma_{s_{\psi}u}& (u(\lambda_{i}),\psi^{\vee})=-1\\
0 & \text{\rm otherwise}.
\end{cases}
\end{align*}
Thus we obtain
\begin{align*}
\sigma_{s_{\alpha_{i}}}\circ \sigma_{u}=
\begin{cases}
\displaystyle\sum_{(u(\lambda_{i}),\alpha_{j}^{\vee})=1}\sigma_{s_{\alpha_{j}}u}+q \sigma_{s_{\psi}u} & (u(\lambda_{i}),\psi^{\vee})=-1\\
\displaystyle\sum_{(u(\lambda_{i}),\alpha_{j}^{\vee})=1}\sigma_{s_{\alpha_{j}}u} & \text{\rm otherwise}.
\end{cases}
\end{align*}
On the other hand, for \(v_{u(\lambda_{i})}\) we have 
\begin{align*}
&(\sum_{j=1}^{l}e_{-\alpha_{j}}+qe_{\psi})\cdot v_{u(\lambda_{i})}\\
=&
\begin{cases}	
\displaystyle\sum_{(u(\lambda_{i}),\alpha_{j}^{\vee})=1} v_{u(\lambda_{i})-\alpha_{j}}+qv_{u(\lambda_{i})+\psi} & (u(\lambda_{i}),\psi^{\vee})=-1\\
\displaystyle\sum_{(u(\lambda_{i}),\alpha_{j}^{\vee})=1} v_{u(\lambda_{i})-\alpha_{j}} & \text{\rm otherwise}
\end{cases}
\\
=&
\begin{cases}
\displaystyle\sum_{(u(\lambda_{i}),\alpha_{j}^{\vee})=1} v_{s_{\alpha_{j}}u(\lambda_{i})}+qv_{s_{\psi}u(\lambda_{i})} & (u(\lambda_{i}),\psi^{\vee})=-1\\
\displaystyle\sum_{(u(\lambda_{i}),\alpha_{j}^{\vee})=1} v_{s_{\alpha_{j}}u(\lambda_{i})} & \text{\rm otherwise}
\end{cases}
\end{align*}
by using the definitions of (\ref{eq2.2}) and (\ref{eq2.4}). Therefore we obtain
\[\sum_{j=1}^{l}e_{-\alpha_{j}}+qe_{\psi}=\sigma_{s_{\alpha_{i}}}\circ.
\]

\fontsize{9pt}{0pt}\selectfont

\em
\noindent
Department of Mathematics\newline
Faculty of Science and Engineering\newline
Waseda University\newline
3-4-1 Okubo, Shinjuku, Tokyo 169-8555\newline
JAPAN

\end{document}